\documentclass{amsart}
\usepackage{enumerate}
\usepackage{amssymb}
\usepackage{mathtools}
\usepackage{dsfont}
\usepackage{amsmath}
\usepackage{color}
\usepackage[
colorlinks=true,
linkcolor=blue,
citecolor=blue,
urlcolor=blue]{hyperref}

\numberwithin{equation}{section}

\newcommand{\normmm}[1]{{\left\vert\kern-0.25ex\left\vert\kern-0.25ex\left\vert #1
   \right\vert\kern-0.25ex\right\vert\kern-0.25ex\right\vert}}

\newtheorem{thm}{Theorem}[section]

\newtheorem{theorem}[thm]{Theorem}
\newtheorem{definition}[thm]{Definition}
\newtheorem{lemma}[thm]{Lemma}
\newtheorem{proposition}[thm]{Proposition}
\newtheorem{problem}[thm]{Problem}
\newtheorem{corollary}[thm]{Corollary}

\newcommand{\E}{\mathbb{E}}
\newcommand{\R}{\mathbb{R}}
\newcommand{\C}{\mathbb{C}}
\newcommand{\Z}{\mathbb{Z}}
\newcommand{\T}{\mathbb{T}}
\newcommand{\Tr}{\operatorname{Tr}}
\newcommand{\sgn}{\operatorname{sgn}}

\newcommand{\sa}{\mathrm{sa}}
\newcommand{\mm}{\Phi_{mix}}
\newcommand{\norm}[2]{\left\lVert #1\right\rVert_{#2}}
\newcommand{\HS}{\mathrm{HS}}
\newcommand{\ran}{\operatorname{ran}}

\begin{document}

\title{The Sharp Scale Inequality for the Schatten Trace Class}

\author[Jiao]{Yong Jiao}
\address{School of Mathematics and Statistics, HNP-LAMA, Central South University, Changsha 410075, China}
\email{jiaoyong@csu.edu.cn}

\author[Luo]{Sijie Luo${}^{*}$}\thanks{$*$ The corresponding author.}
\address{School of Mathematics and Statistics, HNP-LAMA, Central South University, Changsha 410075, China}
\email{sijieluo@csu.edu.cn}

\author[Wu]{Lian Wu}
\address{School of Mathematics and Statistics, HNP-LAMA, Central South University, Changsha 410075, China}
\email{wulian@csu.edu.cn}
\date{}

\begin{abstract}
In this paper, we establish the metric cotype $2$ inequality for the Schatten trace class $S_{1}$ at the optimal order of the scaling parameter. More precisely, for every integer $n\geq 1$, every even integer $m\geq 4$, and every mapping
$f:\mathbb{Z}_{m}^{n}\to S_{1}$, the following holds
\begin{align*}
&\left(\sum_{j=1}^{n}\mathbb{E}_{x}\left\|f\left(x+\frac{m}{2}e_{j}\right)-f(x)\right\|_{1}^{2}\right)^{1/2}\\
\leq&\left(3\sqrt{70}\,m+6\sqrt{n}\right)\left(\mathbb{E}_{x,\varepsilon}\left\|f(x+\varepsilon)-f(x)\right\|_{1}^{2}\right)^{1/2},
\end{align*}
where $x$ and $\varepsilon$ are independently uniform on $\mathbb{Z}_{m}^{n}$ and $\{-1,0,1\}^{n}$, respectively. The order $\sqrt{n}$ is necessary for any fixed metric cotype constant. This settles the optimal scaling problem for the Schatten trace class posed by Mendel and Naor \cite[Section 8]{MN} (see also \cite[Question 3]{GMN}).
\end{abstract}

\subjclass[2020]{Primary 46B07; Secondary 47B10, 46L52}
\keywords{Metric cotype, Schatten class, Ribe program, Banach space}

\maketitle

\section{Introduction}

Rademacher type and cotype are basic invariants of the linear geometry of Banach spaces. Maurey's factorization theory \cite{Maurey1974} and the work of Maurey and Pisier \cite{M-P1976} relate these invariants to linear operators and finite-dimensional subspaces. We refer to Maurey's survey \cite{Maurey2003} for more detail on this topic. Let $X$ be a Banach space, $1\leq p\leq2$, and $2\leq q<\infty$. We recall here that the space $X$ has \emph{(Rademacher) type} $p$ if there is a constant $T<\infty$ such that
\begin{equation}
\label{eq:sc-intro-rademacher-type}
\left(\mathbb{E}_{\delta}\left\|\sum_{j=1}^{n}\delta_jx_j\right\|_X^p\right)^{1/p}
\leq T\left(\sum_{j=1}^{n}\|x_j\|_X^p\right)^{1/p}
\end{equation}
for every finite sequence $x_1,\ldots,x_n\in X$, where $\delta$ is uniform on $\{-1,1\}^n$. It has \emph{(Rademacher) cotype} $q$ if there is a constant $C<\infty$ such that every such sequence satisfies
\begin{equation}
\label{eq:sc-intro-rademacher-cotype}
\left(\sum_{j=1}^{n}\|x_j\|_X^q\right)^{1/q}
\leq C\left(\mathbb{E}_{\delta}\left\|\sum_{j=1}^{n}\delta_jx_j\right\|_X^q\right)^{1/q}.
\end{equation}
The least constants are denoted by $T_p(X)$ and $C_q(X)$, respectively. It is easy to see that every Banach space has type $1$ and cotype $\infty$.

These type and cotype conditions have numerous geometric consequences of Banach spaces. The fundamental result of Kwapie\'n \cite{Kwapien1972} states that $X$ is isomorphic to a Hilbert space if and only if it has type $2$ and cotype $2$. This striking theorem states that the type and cotype conditions essentially characterize Hilbertian spaces. On the other hand, a fundamental theorem of Maurey and Pisier \cite{M-P1976} characterizes type and cotype conditions on a Banach space $X$ in terms of specific finite dimensional subspaces of $X$. More precisely, for $1\leq r\leq\infty$, we say that a Banach space contains $\ell_r^n$ \emph{uniformly} if there is a constant $K<\infty$ such that, for every $n$, there is a linear embedding $u_n:\ell_r^n\to X$ with $\|u_n\|\|u_n^{-1}\|\leq K$. For an infinite-dimensional Banach space $X$, set
\[
p_X=\sup\{p\in[1,2]:T_p(X)<\infty\},\qquad
q_X=\inf\{q\in[2,\infty):C_q(X)<\infty\},
\]
with $\inf\varnothing=\infty$. The Maurey--Pisier theorem asserts that $X$ contains $\ell_{p_X}^n$ and $\ell_{q_X}^n$ uniformly with distortion at most $1+\eta$; see \cite{M-P1976} and \cite[Theorem 6]{Maurey2003}.

A breakthrough was made by Ribe \cite{Ribe} who showed that uniformly homeomorphic Banach spaces are crudely finitely representable in each other. In other words, the metric structure of a Banach space determines its finite dimensional linear structure. Recently, motivated by the developments of nonlinear geometry, a significant program promoted by Naor \cite{NaorSurvey}, called the Ribe program,  asks for characterizations of such linear properties using only distances; see \cite{Bou1986}. Early work on nonlinear type includes Enflo \cite{Enflo1978} and Bourgain, Milman, and Wolfson \cite{BMW1986}. A metric space $(\mathcal M,d_{\mathcal M})$ has \emph{Enflo type} $p$ if there is $T<\infty$ such that every $f:\{-1,1\}^n\to\mathcal M$ satisfies
\begin{equation}
\label{eq:sc-intro-enflo-type}
\mathbb E_{\delta}d_{\mathcal M}(f(\delta),f(-\delta))^p
\leq T^p\sum_{j=1}^{n}\mathbb E_{\delta}
 d_{\mathcal M}(f(\delta),f(\delta^{(j)}))^p,
\end{equation}
where $\delta^{(j)}\coloneqq (\delta_{1},\cdots, \delta_{j-1},-\delta_{j},\delta_{j+1},\cdots,\delta_{n})$ is the $j$-th flip of $\delta$. On a Banach space, applying this inequality to $f(\delta)=\sum_j\delta_jx_j$ gives Rademacher type $p$. Ivanisvili, van Handel, and Volberg \cite{IVV2020} proved the coincidence between Enflo type and Rademacher type for Banach spaces. Since then several new metric type are introduced to investigate the Lipschitz geometry of Banach spaces, we refer to \cite{IS2022} for  the KKL type and to \cite{C-E2023} for the Talagrand type for metric spaces.

On the other hand, the search for a metric characterization of cotype required a different inequality, because the reversing inequality of \eqref{eq:sc-intro-enflo-type} fails even for Hilbert spaces. Mendel and Naor \cite{MN} resolved this problem by using discrete tori and a scaling parameter. For an even integer $m$, write $\mathbb Z_m=\mathbb Z/m\mathbb Z$, and let $e_1,\ldots,e_n$ denote the coordinate vectors. 
\begin{definition}[Mendel--Naor]
A metric space $(\mathcal M,d_{\mathcal M})$ has \emph{metric cotype} $q$ with constant $\Gamma$ if, for every $n\geq1$, there is an even integer $m$ such that every $f:\mathbb Z_m^n\to\mathcal M$ satisfies
\begin{equation}\label{eq:sc-intro-metric-cotype}
\sum_{j=1}^{n}\mathbb E_x d_{\mathcal M}\left(f\left(x+\frac m2e_j\right),f(x)\right)^q\leq\Gamma^qm^q\mathbb E_{x,\varepsilon}d_{\mathcal M}(f(x+\varepsilon),f(x))^q,
\end{equation}
$x$ and $\varepsilon$ are independent and uniform on $\mathbb Z_m^n$ and $\{-1,0,1\}^n$, respectively, and addition on the torus is modulo $m$.
\end{definition}
The infimum of the admissible constants is denoted by $\Gamma_q(\mathcal M)$. The metric cotype theorem \cite[Theorem~1.2]{MN} states that, for every Banach space $X$ and $q\in[2,\infty)$,
\[
\frac{C_q(X)}{2\pi}\leq\Gamma_q(X)\leq90C_q(X).
\]
In particular, Mendel and Naor \cite{MN} proved that the metric cotype $q$ is equivalent to Rademacher cotype $q$ in Banach spaces. Let $m_q(\mathcal M;n,\Gamma)$ be the smallest even integer $m$ for which \eqref{eq:sc-intro-metric-cotype} holds for every $f:\mathbb{Z}^{n}_{m}\to \mathcal{M}$, with value $+\infty$ if no such integer exists. For every metric space containing at least two points, Mendel and Naor's two-point test \cite[Lemma~2.3]{MN} gives
\begin{equation}
\label{eq:sc-intro-scale-lower}
m_q(\mathcal M;n,\Gamma)\geq\frac{n^{1/q}}{\Gamma}.
\end{equation}
Thus $n^{1/q}$ is the smallest possible order. Motivated by the sharp order appears in the power of $n$, Eskenazis, Mendel, and Naor \cite{E-M-N2019} introduced \emph{sharp metric cotype} $q$ if there are $\Gamma,K<\infty$ such that $m_q(\mathcal M;n,\Gamma)\leq Kn^{1/q}$ for every $n$. Such bounds yield obstructions to uniform and coarse embeddings and lower bounds for the distortion of finite grids \cite[Section 7]{MN}. Mendel and Naor also separated the moment exponent from the cotype exponent in the following manner. For $1\leq p\leq q$, let $m_q^{(p)}(\mathcal M;n,\Gamma)$ be the smallest even $m$ such that every $f:\mathbb Z_m^n\to\mathcal M$ satisfies
\begin{equation}\label{eq:sc-intro-moment-variant}
\sum_{j=1}^{n}\mathbb E_x
 d_{\mathcal M}\left(f\left(x+\frac m2e_j\right),f(x)\right)^p\leq\Gamma^pm^pn^{1-p/q}\mathbb E_{x,\varepsilon}d_{\mathcal M}(f(x+\varepsilon),f(x))^p.
\end{equation}
The value is $+\infty$ if no such $m$ exists, and $m_q^{(q)}=m_q$. The following open question raised in \cite[Section~8, item~1]{MN} is one of the fundamental problems in metric cotype theory of Banach spaces.
\begin{problem}[Mendel--Naor]
For arbitrary Banach space $X$ of cotype $q$ and every $1\leq p\leq q$, whether there are constants $\Gamma,K<\infty$, independent of $n$, such that $m_q^{(p)}(X;n,\Gamma)\leq Kn^{1/q}$ for every $n$.
\end{problem}

It is also worthwhile to mention here that the role of sharp scaling also appears in the metric $X_p$ inequalities of Naor and Schechtman \cite{NS2016}. For $2<q<p<\infty$, their inequalities exclude bi-Lipschitz embeddings of $L_q$ into $L_p$; Naor's sharp form \cite{Naor2016} shows that the largest $\theta\in(0,1]$ for which $(L_q,\|x-y\|_q^\theta)$ admits a bi-Lipschitz embedding into $L_p$ is $q/p$.

Mendel and Naor obtained sharp metric cotype for $K$-convex Banach spaces with cotype $q$ \cite[Theorem~4.1]{MN}. However, without this hypothesis, their proof gave the order $n^{2+1/q}$. Giladi, Mendel, and Naor \cite[Theorem~1.2]{GMN} improved the general bound to
\begin{equation}
\label{eq:sc-intro-gmn-scale}
m_q(X;n,\Gamma)\leq Kn^{1+1/q},
\end{equation}
where $\Gamma$ and $K$ depend only on $q$ and $C_q(X)$.

The Schatten trace class is a principal example for which the argument for spaces of nontrivial type does not apply. For $1\leq r<\infty$, the Schatten class $S_r$ consists of compact operators $T$ on $\ell_2$ for which
\[
\|T\|_r=\left(\operatorname{Tr}|T|^r\right)^{1/r}<\infty,
\]
where $|T|=(T^*T)^{1/2}$ and $\operatorname{Tr}$ is the trace on $B(\ell_{2})$. In particular, $S_1$ is the trace class and $S_2$ is the Hilbert--Schmidt class \cite{Simon}. Tomczak-Jaegermann \cite{TJ} proved that $S_r$ has type $\min\{r,2\}$ and cotype $\max\{r,2\}$. In particular, $S_1$ has metric cotype $2$ by the equivalence of cotype and metric cotype for Banach spaces. However, the diagonal operators form an isometric copy of $\ell_1$, hence $S_1$ has trivial type.

Mendel and Naor \cite[Section 8, item 3]{MN} singled out $S_1$ in their questions concerning the removal of $K$-convexity from the embedding consequences of metric cotype. Subsequently, Giladi, Mendel, and Naor \cite[Section 1.3]{GMN} emphasized $L_1$ and $S_1$ as basic cases of the sharp-scaling problem. For the $S_1$ case, \eqref{eq:sc-intro-gmn-scale} gives $m=O(n^{3/2})$, while \eqref{eq:sc-intro-scale-lower} only requires $m\gtrsim\sqrt n$. Closing this gap is necessary to apply the embedding arguments in \cite[Section~7]{MN} to this target. One concrete motivation is the grid-distortion question in \cite[Question~3]{GMN}. Write $[k]_\infty^n=\{0,\ldots,k\}^n$ with the $\ell_\infty$ distance, and let $c_Y(A)$ be the infimum of $\|h\|_{\mathrm{Lip}}\|h^{-1}\|_{\mathrm{Lip}}$ over injective maps $h:A\to Y$. For $Y=S_1$, the question asks whether
\[
c_{S_1}([\lfloor\sqrt n\rfloor]_\infty^n)\geq c\sqrt n
\]
for some $c>0$ independent of $n$. The general cotype estimate in \cite{GMN} only gives a lower bound of order $n^{1/6}$ for these grids.

 Recently, Cheng, Wang, and Xiang \cite{CWX2026} proved sharp metric cotype $2$ for $L_1$, resolving the open question for the commutative case. Their proof uses the representation of a finite $L_1$ metric as a nonnegative sum of cut metrics \cite{D-L1997} and a nonlinear smoothing argument. However, this representation is not available for arbitrary trace-norm distances between matrices, thus handling the case $S_1$ requires more efforts.

In this paper, we establish sharp metric cotype $2$ for $S_1$. More precisely, we prove the following estimate for every even side length $m\geq4$.

\begin{theorem}\label{thm:main}
For every $n\geq1$, every even $m\geq4$, and every
$f:\Z_m^n\to S_1$,
\begin{equation}\label{main sharp}
\begin{split}
&\sum_{j=1}^n\E_{x\in\Z_m^n}\norm{f(x+\tfrac m2e_j)-f(x)}1^2\\
\leq&\bigl(3\sqrt{70}\,m+6\sqrt n\bigr)^{2}\E_{(x,\varepsilon)\in\Z_m^n\times\{-1,0,1\}^{n}}\norm{f(x+\varepsilon)-f(x)}1^2.
\end{split}
\end{equation}
\end{theorem}


Choosing the smallest even $m\geq\max\{4,\sqrt n\}$ and using \eqref{eq:sc-intro-scale-lower}, we obtain
\[
\frac{\sqrt n}{32}\leq m_2(S_1;n,32)\leq4\sqrt n.
\]
This resolves the $X=S_1$, $p=q=2$ case of \cite[Section~8, item~1]{MN}, and the order $\sqrt n$ is optimal. By \cite[Corollary~7.2]{MN}, the sharp scaling also implies that $\ell_r$ admits neither a uniform nor a coarse embedding into $S_1$ for any $r>2$.

We briefly describe key ingredients in our proof. The first new tool is the mixing operators mapping $\Phi_{\mathrm{mix}}$ defined in \eqref{eq:median}. This mapping is invariant under permutations, on three self-adjoint matrices. Furthermore, for fixed matrices $B,C$, we prove the following factorization property of $\Phi_{mix}$ holds
\[
\Phi_{\mathrm{mix}}(A,B,C)-B=V_{B,C}(A)|C-B|^{1/2}+|C-B|^{1/2}V_{B,C}(A)^{*},
\]
and
\[
\|V_{B,C}(A)\|_2\leq\tfrac12\|A-B\|_1^{1/2}.
\]

By the mapping $\Phi_{mix}$ we construct a nonlinear differentiable approximation as follows. Let $F$ be a measurable function defined on $\mathbb{R}^{n}$ with finite range in the space of $d\times d$ self-adjoint matrices, and let $W^{(1)},W^{(2)},W^{(3)}$ be independent vectors whose coordinates are independent with density $\psi(t)=84|t|^5(1-|t|)^2\mathbf1_{[-1,1]}(t)$, taken from \cite[Section 2.2]{CWX2026}. We define the following mapping
\[
g(z)=\mathbb E\Phi_{\mathrm{mix}}\bigl(F(z+W^{(1)}),F(z+W^{(2)}),F(z+W^{(3)})\bigr).
\]
It will be derived in Section \ref{approximate section} that $g\in C^{1}(\mathbb{R}^{n};M_{d})$ such that $g$ approximates $F$ in the following manner
\[
\begin{aligned}
\|g(z)-F(z)\|_1&\leq3\mathbb E_W\|F(z+W)-F(z)\|_1,\\
\left(\sum_{j=1}^n\|\partial_jg(z)\|_1^2\right)^{1/2}
&\leq6\sqrt{70}\,\mathbb E_W\|F(z+W)-F(z)\|_{1},
\end{aligned}
\]
where $W$ has the same law as each $W^{(i)}$. Applying these bounds to the periodic extension $F(z)=f(\lfloor z\rfloor)$ for each $z\in \mathbb{R}^{n}$, with integer coordinates reduced modulo $m$, yields \eqref{main sharp}. Finally, the Hermitian dilation argument and finite-rank approximation of compact operator yield the full result for the Schatten trace class $S_1$.

The paper is organized as follows. Section \ref{preliminaries and matrix analysis} collects necessary background of matrix analysis that will be used throughout this paper. Section \ref{three matrices} is devoted to introducing and analyzing the mix three matrices mapping $\Phi_{mix}$. Basic factorization of $\Phi_{mix}$ will be stated and proved in detail. Section \ref{expectation section} and Section \ref{approximate section} establishes the expectation estimates of the mix three matrices mapping $\Phi_{mix}$ and combine the specific probability density $\psi$ with $\Phi_{mix}$ to produce a smoothing approximation, respectively. We conclude this paper, in Section \ref{full proof}, to give a proof of the main result Theorem \ref{thm:main} and the optimal order of the scaling parameter.

All notations in this paper are standard. For any $x\in \mathbb{R}$, we let $\lfloor x\rfloor$ denotes the floor of $x$, and $\lfloor z\rfloor\coloneqq (\lfloor z_{1}\rfloor,\cdots, \lfloor z_{n}\rfloor)$ for $z\in \mathbb{R}^{n}$. For $d\geq1$, let $M_{d}=M_{d}(\C)$ be $*$-algebra of all $d\times d$ complex matrices and
$M_{d}^{\sa}=\{A\in M_{d}:A=A^{*}\}$ be the self-adjoint part of $M_{d}$. For a given random variable $W$, we use $\mathbb{E}_{W}$ to stand for the expectation with respect to $W$.

\section{Preliminaries}\label{preliminaries and matrix analysis}

This section is devoted to providing estimates invoking Schatten $p$-norm of operators. We first specify the matrix conventions and prove the norm estimates used below. These are standard finite-dimensional facts from spectral theory and the theory of unitarily invariant norms; see
\cite[Chapters I and IV]{Bhatia} and \cite{Penrose, Simon}.

For two matrices $S$ and $T$, we recall the operator order $S\leq T$ as follows
\[
\langle(T-S)v,v\rangle\geq0\quad\text{for every }v\in\C^{d}.
\]
A matrix $T$ is called positive if it is self-adjoint and $T\geq 0$ with respect to the operator order. If
$S=\sum_{\lambda\in\sigma(S)}\lambda E_{\lambda}$ is the spectral
decomposition of a self-adjoint matrix, set
\[
 |S|=\sum_{\lambda\in\sigma(S)}|\lambda|E_{\lambda},\qquad
 \sgn(S)=\sum_{\substack{\lambda\in\sigma(S)\\\lambda\neq0}}
                   \frac{\lambda}{|\lambda|}E_{\lambda}.
\]
Thus $S=\sgn(S)|S|=|S|\sgn(S)$ is the polar decomposition of $S$. For a positive matrix $D$, define
\begin{equation}
 P_{D}=\sum_{\lambda>0}E_{\lambda},\qquad
 D^{\dagger}=\sum_{\lambda>0}\lambda^{-1}E_{\lambda}.
 \label{eq:spectralinverse}
\end{equation}
Here the sums range over the positive eigenvalues of $D$. The matrix $P_{D}$ is the orthogonal projection onto
$\ran D$, and
\[
 D^{\dagger}D=DD^{\dagger}=P_{D}.
\]

The Hilbert--Schmidt inner product and the operator norm are
\[
 \langle A,B\rangle_{\HS}=\Tr(B^{*}A),\qquad
 \norm{A}{2}=\left(\Tr(A^{*}A)\right)^{1/2},\qquad
 \norm{A}{\infty}=\sup_{\norm{v}{\C^{d}}=1}\norm{Av}{\C^{d}}.
\]
In particular, $\norm{\sgn(S)}{\infty}\leq1$. We shall use the following noncommutative H\"older inequality. For each $1\leq p,~q\leq \infty$ with $\frac{1}{p}+\frac{1}{q}=\frac{1}{r}$, the following inequality holds for all $A$, $B\in M_{d}$,
\begin{equation} \label{eq:holder}
 \norm{AB}{r}\leq\norm{A}{p}\norm{B}{q}.
\end{equation}

We also need the following support implication. If $0\leq S\leq D$
and $v\in\ker D$, then
\[
 0\leq\norm{S^{1/2}v}{\C^{d}}^{2}
 =\langle Sv,v\rangle\leq\langle Dv,v\rangle=0.
\]
It follows that $\ker D\subseteq\ker S$. If instead $S=S^{*}$ and
$|S|\leq D$, apply this implication to $|S|$ and use
$\ker|S|=\ker S$. Thus $S(I_{d}-P_{D})=0$. Taking adjoints yields
$(I_{d}-P_{D})S=0$, so
\begin{equation}
 SP_{D}=P_{D}S=S\qquad\text{whenever }S=S^{*}\text{ and }|S|\leq D.
 \label{eq:support}
\end{equation}

We conclude this section by recalling the Banach space-valued Lebesgue space. Let $B$ be a Banach space and $(\Omega,\mathcal{F}, \mu)$ be a $\sigma$-finite measure space. The $L_{p}$ Bochner--Lebesgue space $L_{p}(\Omega;B)$ is the space consisting of all Bochner measurable function $f:\Omega\to B$ with
\[
\|F\|_{L_{p}(B)}\coloneqq \left(\int_{\Omega} \|F(\theta)\|^{p}_{B}~d\mu(\theta)\right)^{1/p}<\infty,
\]
for $1\leq p<\infty$ and
\[
\|F\|_{L_{\infty}(B)}\coloneqq \operatorname{ess \, sup}\|F(\theta)\|_{B}<\infty.
\]
Then $L_{p}(\Omega; B)$ is a Banach space for all $1\leq p \leq \infty$.
\section{The three matrices operation}\label{three matrices}
This section is devoted to introducing the mixing operator mapping $\Phi_{mix}$ and exploring its basic properties. We begin with the following basic estimate on operator norm of the composition of tree matrices.

\begin{lemma}\label{lem:cross}
If $S,T,D\in M_{d}$ are positive and $S+T\leq D$, then
\begin{equation}
 \norm{S^{1/2}D^{\dagger}T^{1/2}}{\infty}\leq\frac12.
 \label{eq:cross}
\end{equation}
\end{lemma}
\begin{proof}
Let
\[
 L=\bigl[\,(D^{\dagger})^{1/2}S^{1/2}
            \quad(D^{\dagger})^{1/2}T^{1/2}\,\bigr],
\]
viewed as an operator from $\C^{d}\oplus\C^{d}$ to $\C^{d}$. Then
\[
 LL^{*}=(D^{\dagger})^{1/2}(S+T)(D^{\dagger})^{1/2}
 \leq(D^{\dagger})^{1/2}D(D^{\dagger})^{1/2}
 =P_{D}\leq I_{d}.
\]
Consequently $\norm{L}{\infty}\leq1$ and
$0\leq L^{*}L\leq I_{2d}$. The eigenvalues of $L^{*}L$ lie in $[0,1]$,
so
\[
 \norm{L^{*}L-\tfrac12I_{2d}}{\infty}\leq\frac12.
\]
For $u,v\in\C^{d}$, block multiplication gives
\[
\left|\left\langle S^{1/2}D^{\dagger}T^{1/2}v,u\right\rangle\right|=\left|\left\langle\left(L^{*}L-\tfrac12I_{2d}\right)(0,v),(u,0)\right\rangle\right|\leq\frac12\norm{u}{\C^{d}}\norm{v}{\C^{d}}.
\]
Taking the supremum over unit vectors proves \eqref{eq:cross}.
\end{proof}

We now introduce the following \emph{three operators symmetrization mapping}. For $A,B,C\in M_{d}^{\sa}$, define
\[
 D(A,B,C)=|A-B|+|B-C|+|C-A|,
 \qquad \operatorname{Re}T=\frac{T+T^{*}}2.
\]
The positive matrix $D(A,B,C)$ is unchanged by permuting its three
arguments. In formulas involving a fixed triple we write simply
$D=D(A,B,C)$. We now define the \emph{mixing operators mapping} as follows
\begin{equation}
 \begin{split}
 \mm(A,B,C)=A+\operatorname{Re}\bigl[
     &(B-A)D^{\dagger}|C-A|+(C-A)D^{\dagger}|B-A|\bigr].
 \end{split}
 \label{eq:median}
\end{equation}
If $D=0$, then
$A=B=C$ and \eqref{eq:median} gives $\mm(A,A,A)=A$.

\begin{lemma}\label{lem:matrix}
The operation $\mm$ is invariant under every permutation of its arguments.
For each fixed $B,C\in M_{d}^{\sa}$ there exists a map
$V_{B,C}:M_{d}^{\sa}\to M_{d}$ such that, for every $A\in M_{d}^{\sa}$,
\begin{align}
 \mm(A,B,C)
 &=B+V_{B,C}(A)|C-B|^{1/2}
      +|C-B|^{1/2}V_{B,C}(A)^{*},
 \label{eq:factor}\\
 \norm{V_{B,C}(A)}{2}
 &\leq\frac12\norm{A-B}{1}^{1/2}.
 \label{eq:Zbound}
\end{align}
In particular,
\begin{equation}
 \norm{\mm(A,B,C)-B}{1}
 \leq\sqrt{\norm{A-B}{1}\norm{C-B}{1}}.
 \label{eq:medianbound}
\end{equation}
\end{lemma}

\begin{proof}
Interchanging $B,C$ in \eqref{eq:median} interchanges its two summands
and leaves $D$ unchanged. To interchange $A,B$, subtract the two
defining formulas. Since $|A-B|=|B-A|$, this gives
\[
 \begin{aligned}
 &\mm(A,B,C)-\mm(B,A,C)\\
 =&A-B+\operatorname{Re}\bigl[
       (B-A)D^{\dagger}\left(|C-A|+|C-B|\right)\\
 &+\left((C-A)-(C-B)\right)D^{\dagger}|A-B|\bigr]\\
 =&A-B+\operatorname{Re}\left[
       (B-A)D^{\dagger}
       \left(|C-A|+|C-B|+|A-B|\right)\right]\\
 =&A-B+\operatorname{Re}\left[(B-A)D^{\dagger}D\right]=0.
 \end{aligned}
\]
The last equality follows from $D^{\dagger}D=P_{D}$ and
$(B-A)P_{D}=B-A$, by \eqref{eq:support}. The transpositions of the
first two and the last two arguments generate every permutation, proving
symmetry.

For fixed $B,C$, define, with $D=D(A,B,C)$,
\begin{equation}\label{eq:Z}
 V_{B,C}(A)=\frac12\bigl[(A-B)D^{\dagger}|C-B|^{1/2}+|A-B|D^{\dagger}|C-B|^{1/2}\sgn(C-B)\bigr].
\end{equation}
The spectral functions $|C-B|^{1/2}$ and $\sgn(C-B)$ commute and satisfy
\[
 |C-B|^{1/2}\sgn(C-B)|C-B|^{1/2}=C-B.
\]
Therefore
\[
V_{B,C}(A)|C-B|^{1/2}
 =\frac12\bigl[(A-B)D^{\dagger}|C-B|+|A-B|D^{\dagger}(C-B)\bigr].
\]
Adding the adjoint, using $\operatorname{Re}(T^{*})=\operatorname{Re}T$ and the definition in \eqref{eq:median} proves \eqref{eq:factor}.

By \eqref{eq:cross}, applied to $|A-B|+|C-B|\leq D$,
\[
 \norm{|A-B|^{1/2}D^{\dagger}|C-B|^{1/2}}{\infty}\leq\frac12.
\]
By \eqref{eq:Z} and \eqref{eq:holder}, we have
\begin{align*}
\norm{V_{B,C}(A)}{2}&\leq \frac{1}{2}\norm{(A-B)D^{\dagger}|C-B|^{1/2}}{2}\\
&\quad+\frac{1}{2}\norm{|A-B|D^{\dagger}|C-B|^{1/2}\sgn(C-B)}{2}\\
&\leq \norm{|A-B|D^{\dagger}|C-B|^{1/2}}{2}\\
&\leq \||A-B|^{1/2}\|_{2}\cdot\norm{|A-B|^{1/2}D^{\dagger}|C-B|^{1/2}}{\infty}\\
&\leq \frac{1}{2}\norm{|A-B|^{1/2}}{2}=\frac{1}{2} \norm{A-B}{1}^{1/2},
\end{align*}
which verifies \eqref{eq:Zbound}. Applying \eqref{eq:holder} to
\eqref{eq:factor} now gives
\[
 \norm{\mm(A,B,C)-B}{1}
 \leq2\norm{V_{B,C}(A)}{2}\norm{|C-B|^{1/2}}{2}
 \leq\sqrt{\norm{A-B}{1}\norm{C-B}{1}},
\]
as required.
\end{proof}

\section{Expectation of the matrix operation}\label{expectation section}

All expectations of finite-dimensional matrices are taken entry by entry.
The notation $\E_X$ means integration with respect to $X$ alone, with any
other displayed variables held fixed.

\begin{proposition}\label{prop:probability}
Let $X$ be a random matrix in $M_d^{\sa}$ with finite range, and let $B,C$
be independent copies of $X$. For every $A_0\in M_d^{\sa}$,
\begin{equation}
 \norm{\E\mm(X,B,C)-A_0}{1}
 \leq 3\E\norm{X-A_0}{1}.
 \label{eq:approx}
\end{equation}
Suppose that $X$ and real random variables $\xi_1,\ldots,\xi_n$ are defined
on the same probability space, that $B,C$ are independent of that space
and of each other, and that, for some $\kappa\geq0$,
\begin{equation}
 \E\xi_j=0,\qquad \E(\xi_i\xi_j)=\kappa^2\delta_{ij}.
 \label{eq:scores}
\end{equation}
Here $\delta_{ij}$ is $1$ if $i=j$ and $0$ otherwise. Then
\begin{equation}
 \left(\sum_{j=1}^n
   \norm{\E\bigl[\xi_j\mm(X,B,C)\bigr]}{1}^2\right)^{1/2}
 \leq\kappa\E\norm{X-B}{1}.
 \label{eq:scorebound}
\end{equation}
No independence between $X$ and the $\xi_j$ is required.
\end{proposition}

\begin{proof}
By independence of $X$ and $C$ for fixed $B$, followed by the scalar
Cauchy--Schwarz inequality,
\begin{align}
 \E\sqrt{\norm{X-B}{1}\norm{C-B}{1}}
 &=\E_B\left(\E_X\sqrt{\norm{X-B}{1}}\right)^2\leq\E_B\E_X\norm{X-B}{1}.
 \label{eq:meanroot}
\end{align}
Consequently, \eqref{eq:medianbound} and the triangle inequality give
\begin{align*}
 \norm{\E\mm(X,B,C)-A_0}{1}
 &\leq\E\norm{\mm(X,B,C)-B}{1}+\E\norm{B-A_0}{1}\\
 &\leq\E\norm{X-B}{1}+\E\norm{B-A_0}{1}\\
 &\leq\E\norm{X-A_0}{1}+2\E\norm{B-A_0}{1}
 =3\E\norm{X-A_0}{1}.
\end{align*}
This proves \eqref{eq:approx}.

Assume without loss of generality that $\kappa>0$. For fixed
$B,C$, use $V_{B,C}(X)$ from Lemma~\ref{lem:matrix} and define only the
coefficients needed for the orthogonality argument:
\[
 Y_j(B,C)=\E_\omega\bigl[\xi_j(\omega)V_{B,C}(X(\omega))\bigr],
 \qquad 1\leq j\leq n.
\]
The subscript $\omega$ means that $B,C$ remain fixed. Expanding the
Hilbert--Schmidt norm, using that the $\xi_j$ are real and
$\E(\xi_i\xi_j)=\kappa^2\delta_{ij}$, gives
\begin{align*}
 0&\leq\E_\omega\left\lVert
 V_{B,C}(X)-\kappa^{-2}\sum_{j=1}^n\xi_jY_j(B,C)
 \right\rVert_2^2\\
 &=\E_\omega\norm{V_{B,C}(X)}{2}^2
   -\kappa^{-2}\sum_{j=1}^n\norm{Y_j(B,C)}{2}^2.
\end{align*}
In particular, \eqref{eq:Zbound} implies
\begin{equation}
 \sum_{j=1}^n\norm{Y_j(B,C)}{2}^2
 \leq\kappa^2\E_X\norm{V_{B,C}(X)}{2}^2
 \leq\frac{\kappa^2}{4}\E_X\norm{X-B}{1}.
 \label{eq:bessel}
\end{equation}
The equality defining $Y_j$ averages $\xi_j$ jointly with $X$; the last
expectation in \eqref{eq:bessel} only uses the marginal distribution of $X$.

Since $\E\xi_j=0$ and $\xi_j$ is real, \eqref{eq:factor} yields
\[
 \E_\omega[\xi_j\mm(X,B,C)]
 =Y_j(B,C)|C-B|^{1/2}+|C-B|^{1/2}Y_j(B,C)^*.
\]
For any integrable finite family of random matrices $(T_j)_{j=1}^n$,
\begin{equation}
 \left(\sum_j\norm{\E T_j}{1}^2\right)^{1/2}
 \leq\E\left(\sum_j\norm{T_j}{1}^2\right)^{1/2}.
 \label{eq:familytriangle}
\end{equation}
Indeed, $\norm{\E T_j}{1}\leq\E\norm{T_j}{1}$, and the Euclidean
triangle inequality for integrals gives the remaining inequality. Applying
\eqref{eq:familytriangle}, then \eqref{eq:holder} and \eqref{eq:bessel},
we obtain
\begin{align*}
& \left(\sum_{j=1}^n\norm{\E[\xi_j\mm(X,B,C)]}{1}^2\right)^{1/2}\\
 \leq& 2\E_{B,C}\left[
   \norm{|C-B|^{1/2}}{2}
   \left(\sum_{j=1}^n\norm{Y_j(B,C)}{2}^2\right)^{1/2}\right]\\
 \leq&\kappa\E_{B,C}
   \sqrt{\norm{C-B}{1}\,\E_X\norm{X-B}{1}}\\
 \leq&\kappa\E_B
   \sqrt{\bigl(\E_C\norm{C-B}{1}\bigr)
               \bigl(\E_X\norm{X-B}{1}\bigr)}\\
 =&\kappa\E_B\E_X\norm{X-B}{1}.
\end{align*}
For the penultimate line, Cauchy--Schwarz is applied only to $C$, with
$B$ fixed. The last equality holds because $C$ and $X$ have the same
distribution.
\end{proof}

\section{The differentiable approximation}\label{approximate section}

\begin{lemma}\label{lem:kernel}
The function
\begin{equation}
 \psi(t)=84|t|^5(1-|t|)^2\mathbf1_{[-1,1]}(t)
 \label{eq:kernel}
\end{equation}
is an even probability density in $C_c^1(\R)$, and
\begin{equation}
 \int_\R |t|\psi(t)\,dt=\frac23,
 \qquad
 \int_{\{\psi>0\}}\frac{\psi'(t)^2}{\psi(t)}\,dt=70.
 \label{eq:kernelmoments}
\end{equation}
Let $W=(W_1,\ldots,W_n)$ have independent coordinates with density $\psi$.
For $1\leq j\leq n$, define
\begin{equation}
 \xi_j(w)=
 \begin{cases}
 -\psi'(w_j)/\psi(w_j),&\psi(w_j)>0,\\
 0,&\psi(w_j)=0.
 \end{cases}
 \label{eq:xidef}
\end{equation}
Then
\begin{equation}
 \E\xi_j(W)=0,\qquad \E[\xi_i(W)\xi_j(W)]=70\delta_{ij}.
 \label{eq:kernelcovariance}
\end{equation}
If $U$ is independently uniform on $[0,1)^n$, then
$\lfloor U+W\rfloor$ is uniform on $\{-1,0,1\}^n$. Here the floor
function is applied to each coordinate.
\end{lemma}

\begin{proof}
The function and its derivative vanish at $\pm1$, and the factor $|t|^5$
makes the derivative continuous at zero. Hence $\psi\in C_c^1(\R)$.
For $0<t<1$,
\[
 \psi'(t)=84t^4(1-t)(5-7t),\qquad
 \frac{\psi'(t)^2}{\psi(t)}=84t^3(5-7t)^2.
\]
Evenness and integration give
\begin{align*}
 \int_\R\psi(t)\,dt
 &=168\left(\frac16-\frac27+\frac18\right)=1,\\
 \int_\R|t|\psi(t)\,dt
 &=168\left(\frac17-\frac28+\frac19\right)=\frac23,\\
 \int_{\{\psi>0\}}\frac{\psi'(t)^2}{\psi(t)}\,dt
 &=168\left(\frac{25}{4}-\frac{70}{5}+\frac{49}{6}\right)=70.
\end{align*}
Since $\{\psi(W_j)=0\}$ has probability zero,
\[
 \E\xi_j(W)=-\int_\R\psi'(t)\,dt=0,
 \qquad \E\xi_j(W)^2=70.
\]
For $i\ne j$, independence yields
$\E[\xi_i(W)\xi_j(W)]=\E\xi_i(W)\E\xi_j(W)=0$, which proves the assertion \eqref{eq:kernelcovariance}.

Fix $w\in[-1,1]$. The subsets of $[0,1)$ on which
$\lfloor u+w\rfloor$ equals $1,-1,0$ have lengths, respectively,
\[
 \max\{w,0\},\qquad \max\{-w,0\},\qquad 1-|w|.
\]
Therefore, by evenness and \eqref{eq:kernelmoments},
\[
 \mathbb{P}(\lfloor U_j+W_j\rfloor=1)
 =\mathbb{P}(\lfloor U_j+W_j\rfloor=-1)=\frac12\E|W_j|=\frac13,
\]
and the remaining probability is $1-\E|W_j|=1/3$.
The pairs $(U_j,W_j)$ are independent across $j$. Thus, for every
$\varepsilon\in\{-1,0,1\}^n$,
\begin{equation}
 \mathbb{P}(\lfloor U+W\rfloor=\varepsilon)
 =\prod_{j=1}^n\mathbb{P}(\lfloor U_j+W_j\rfloor=\varepsilon_j)=3^{-n}.
 \label{eq:productrounding}
\end{equation}
\end{proof}

Let $F:\R^n\to M_d^{\sa}$ be measurable and have finite range. Let $W^{(1)}$, $W^{(2)}$, $W^{(3)}$ be three independent copies of the vector $W$ in
Lemma \ref{lem:kernel}. We define
\begin{equation}\label{eq:smoothing}
 g(z)=\E\mm\bigl(F(z+W^{(1)}),F(z+W^{(2)}),F(z+W^{(3)})\bigr).
\end{equation}
It is easy to see that if $F$ has period $m$ in each coordinate, then so does $g$.

\begin{proposition}\label{prop:smoothing}
The function $g$ in \eqref{eq:smoothing} belongs to
$C^1(\R^n;M_d^{\sa})$. For every $z\in\R^n$, the following two inequalities hold
\begin{equation}\label{eq:localapprox}
 \norm{g(z)-F(z)}{1}\leq3\E_W\norm{F(z+W)-F(z)}{1},
\end{equation}
and
\begin{equation}\label{eq:localgradient}
\left(\sum_{j=1}^n\norm{\partial_jg(z)}{1}^2\right)^{1/2}\leq6\sqrt{70}\,\E_W\norm{F(z+W)-F(z)}{1}.
\end{equation}
\end{proposition}

\begin{proof}
Suppose without loss of generality that $\mathrm{rang}(F)=\{A_{1},\cdots,A_{q}\}\subseteq M^{\sa}_{d}$. Define density as follows
\[
 \rho(w)=\prod_{i=1}^n\psi(w_i),\qquad
 p_a(z)=\int_{\R^n}\mathbf1_{\{F(y)=A_a\}}\rho(y-z)\,dy,\qquad a\in \{1,\cdots, q\}.
\]
Changing variables gives $p_a(z)=\mathbb{P}(F(z+W)=A_a)$ for each $a\in \{1,\cdots, q\}$. In particular,
$p_a\geq0$ and $\sum_{a\in \{1,\cdots,q\}} p_a=1$.

By independence of the three copies and the fact that $F$ has finite range, we have
\begin{equation}\label{summation representation 1}
g(z)=\sum_{a,b,c=1}^q p_a(z)p_b(z)p_c(z)\mm(A_a,A_b,A_c).
\end{equation}

We next prove that $p_{a}\in C^{1}(\mathbb{R}^{n})$
for every $a\in\{1,\ldots,q\}$.
Recall that
\[
p_{a}(z)=
\int_{\mathbb{R}^{n}}
\mathbf{1}_{\{F(y)=A_{a}\}}\rho(y-z)\,dy,
\qquad
\rho(w)=\prod_{i=1}^{n}\psi(w_{i}).
\]
By Lemma \ref{lem:kernel}, $\psi\in C_{c}^{1}(\mathbb{R})$ is supported in $[-1,1]$. Consequently, $\rho=\prod_{i=1}^n\psi(w_i)\in C_{c}^{1}(\mathbb{R}^{n})$, and
\[
\partial_{j}\rho(w)=\psi^{\prime}(w_{j})\prod_{\substack{i=1\\i\neq j}}^{n}\psi(w_{i}),
\qquad 1\leq j\leq n.
\]
In particular, both $\rho$ and its first partial derivatives vanish outside $[-1,1]^{n}$.

Fix $z_{0}\in\mathbb{R}^{n}$ and
$j\in\{1,\ldots,n\}$, and consider the cube
\[
K=z_{0}+[-2,2]^{n}.
\]
For $0<|v|<1$, we have
\[
\frac{p_{a}(z_{0}+ve_{j})-p_{a}(z_{0})}{v}=\int_{\mathbb{R}^{n}}\mathbf{1}_{\{F(y)=A_{a}\}}\frac{\rho(y-z_{0}-ve_{j})-\rho(y-z_{0})}{v}\,dy.
\]
Since $\rho$ is continuously differentiable,
the fundamental theorem of calculus gives
\[
\frac{\rho(y-z_{0}-ve_{j})-\rho(y-z_{0})}{v}=-\int_{0}^{1}
\partial_{j}\rho(y-z_{0}-tve_{j})\,dt.
\]
For every $t\in[0,1]$ and $|v|<1$, the support of the function $y\mapsto \partial_{j}\rho(y-z_{0}-tve_{j})$ is contained in $z_{0}+tve_{j}+[-1,1]^{n}
\subseteq K$. Therefore we have
\[
\left|
\mathbf{1}_{\{F(y)=A_{a}\}}\left(\frac{\rho(y-z_{0}-ve_{j})-\rho(y-z_{0})}{v}\right)\right|\leq\left\|\partial_{j}\rho\right\|_{\infty}\mathbf{1}_{K}(y),
\]
where $K$ is a compact subset in $\mathbb{R}^{n}$. Moreover, for every fixed $y\in\mathbb{R}^{n}$,
\[
\lim_{v\to0}
\frac{\rho(y-z_{0}-ve_{j})-\rho(y-z_{0})}{v}=-\partial_{j}\rho(y-z_{0}),
\]
the dominated convergence entails that
\[
\partial_{j}p_{a}(z_{0})=-\int_{\mathbb{R}^{n}}\mathbf{1}_{\{F(y)=A_{a}\}}\partial_{j}\rho(y-z_{0})\,dy.
\]
Since $z_{0}$ and $j$ were arbitrary, we obtain
\begin{equation}\label{integral representation 1}
\partial_{j}p_{a}(z)=-\int_{\mathbb{R}^{n}}\mathbf{1}_{\{F(y)=A_{a}\}}\partial_{j}\rho(y-z)\,dy
\end{equation}
for every $z\in\mathbb{R}^{n}$ and
$j\in\{1,\ldots,n\}$.

It remains to verify the continuity of these partial derivatives. Suppose that $\{z^{(k)}\}_{k=1}^{\infty}\subseteq \mathbb{R}^{n}$ such that $z^{(k)}\to z_{0}$. For all sufficiently large $k$, we have
$z^{(k)}\in z_{0}+[-1,1]^{n}$, and therefore
\[
\left|\mathbf{1}_{\{F(y)=A_{a}\}}\partial_{j}\rho(y-z^{(k)})\right|\leq\left\|\partial_{j}\rho\right\|_{\infty}\mathbf{1}_{K}(y).
\]
On the other hand, the continuity of
$\partial_{j}\rho$ implies that
\[
\partial_{j}\rho(y-z^{(k)})
\longrightarrow
\partial_{j}\rho(y-z_{0})
\]
for every $y\in\mathbb{R}^{n}$. We apply the dominated convergence theorem again to obtain
\[
\lim_{k\to\infty}
\partial_{j}p_{a}(z^{(k)})
=
\partial_{j}p_{a}(z_{0}).
\]
Thus every first partial derivative of
$p_{a}$ is continuous, proving that
$p_{a}\in C^{1}(\mathbb{R}^{n})$.

We now express these derivatives in terms
of the random variables $\xi_{j}$.
Since $\psi$ is nonnegative and continuously
differentiable, $\psi'(t)=0$ whenever
$\psi(t)=0$. Hence, by the definition of
$\xi_{j}$ in (5.3), we have
\[
-\partial_{j}\rho(w)=\xi_{j}(w)\rho(w),
\qquad w\in\mathbb{R}^{n}.
\]
Furthermore, we have $\mathbb{E}_{W}|\xi_{j}(W)|=\int_{\mathbb{R}^{n}}|\xi_{j}(w)|\rho(w)\,dw=
\int_{\mathbb{R}}|\psi'(t)|\,dt<\infty$, where the last inequality follows from
$\psi\in C_{c}^{1}(\mathbb{R})$.
Changing variables $y=z+w$ in the integral representation \eqref{integral representation 1}, we obtain
\begin{equation}
\begin{aligned}
\partial_{j}p_{a}(z)
&=
-\int_{\mathbb{R}^{n}}
\mathbf{1}_{\{F(y)=A_{a}\}}
\partial_{j}\rho(y-z)\,dy\\
&=
\int_{\mathbb{R}^{n}}
\mathbf{1}_{\{F(z+w)=A_{a}\}}
\xi_{j}(w)\rho(w)\,dw\\
&=
\mathbb{E}_{W}\left[
\xi_{j}(W)
\mathbf{1}_{\{F(z+W)=A_{a}\}}
\right].
\end{aligned}
\end{equation}

Since each $p_{a}$ belongs to $C^{1}(\mathbb{R}^{n})$, \eqref{summation representation 1} implies that $g\in C^{1}(\mathbb{R}^{n};M_{d}^{\mathrm{sa}})$.
Applying the product rule to \eqref{summation representation 1}, we obtain
\begin{align*}
\partial_{j}g(z)&=\sum_{a,b,c=1}^{q}\Big[(\partial_{j}p_{a}(z))p_{b}(z)p_{c}(z)+p_{a}(z)(\partial_{j}p_{b}(z))p_{c}(z)\\
&\qquad+p_{a}(z)p_{b}(z)(\partial_{j}p_{c}(z))
\Big]\Phi_{\mathrm{mix}}(A_{a},A_{b},A_{c}).
\end{align*}
By Lemma \ref{lem:matrix}, $\Phi_{\mathrm{mix}}$ is invariant
under every permutation of its arguments. Hence,
\begin{equation}\label{summation representation 2}
\partial_{j}g(z)=3\sum_{a,b,c=1}^{q}(\partial_{j}p_{a}(z))p_{b}(z)p_{c}(z)\Phi_{\mathrm{mix}}(A_{a},A_{b},A_{c}).
\end{equation}

Finally, by \eqref{integral representation 1} and the independence
of $W^{(1)},W^{(2)},W^{(3)}$, we have
\begin{equation}\label{representation of partial derivative}
\begin{split}
&(\partial_{j}p_{a}(z))p_{b}(z)p_{c}(z)\\
=&\mathbb{E}\Big[
\xi_{j}(W^{(1)})
\mathbf{1}_{\{F(z+W^{(1)})=A_{a}\}}\cdot\mathbf{1}_{\{F(z+W^{(2)})=A_{b}\}}\mathbf{1}_{\{F(z+W^{(3)})=A_{c}\}}\Big]
\end{split}
\end{equation}
Substituting \eqref{representation of partial derivative} into \eqref{summation representation 2} yields
\begin{equation}\label{derivative}
\partial_{j}g(z)=3\mathbb{E}\Big[\xi_{j}(W^{(1)})\Phi_{\mathrm{mix}}\big(F(z+W^{(1)}),F(z+W^{(2)}),F(z+W^{(3)})\big)\Big].
\end{equation}
For each fixed $z$, we apply \eqref{eq:approx} to the three independent
matrices $F(z+W^{(1)}),F(z+W^{(2)}),F(z+W^{(3)})$, with $A_0=F(z)$. This verifies \eqref{eq:localapprox}. Combining \eqref{eq:scorebound},  \eqref{eq:kernelcovariance} with \eqref{derivative} gives
\begin{align*}
 \left(\sum_{j=1}^n\norm{\partial_jg(z)}{1}^2\right)^{1/2}
 &\leq3\sqrt{70}\,
 \E\norm{F(z+W^{(1)})-F(z+W^{(2)})}{1}\\
 &\leq6\sqrt{70}\,\E_W\norm{F(z+W)-F(z)}{1},
\end{align*}
where the last step follows from the triangle inequality and the same distribution of $W^{(1)}$ and $W^{(2)}$. This proves \eqref{eq:localgradient}. Finally, any identity $F(z+me_j)=F(z)$ for all $z$ holds in each argument of \eqref{eq:smoothing}, and gives $g(z+me_j)=g(z)$.
\end{proof}

\section{The metric cotype $2$ with optimal scaling parameter}\label{full proof}

We now provide a proof of the main theorem of this paper.

\begin{proof}[Proof of Theorem \ref{thm:main}]
{\bf Step I. Finite dimensional self-adjoint case.} Suppose that $f:\Z_m^n\to M_d^{\sa}$ and equip $\T_m^n=(\R/m\Z)^n$ with normalized Lebesgue measure
$d\nu(z)=m^{-n}\,dz$ on $[0,m)^n$. In the sequel, integer vectors in the argument of $f$ are reduced coordinatewise modulo $m$. Define
\begin{equation}\label{eq:cells}
 F(z)=f(\lfloor z\rfloor)=f(\lfloor z_{1}\rfloor,\cdots, \lfloor z_{n}\rfloor),\qquad z\in\R^n,
\end{equation}
It is clear that $F$ is measurable and has finite range such that
\[
F(z+me_j)=F(z),\qquad z\in \mathbb{R}^{n}.
\]
In the sequel, we view $F$ as the function defined on $\mathbb{T}^{n}_{m}$ taking values in $M^{\sa}_{d}$. For any integrable function $h:\mathbb{T}^{n}_{m}\to \mathbb{R}$, by the decomposition of the cube $\mathbb{T}^{n}_{m}$, we have the following
\begin{equation}\label{eq:cellintegration}
\begin{split}
 \int_{\T_m^n}h(z)\,d\nu(z)&=\frac{1}{m^{n}}\int_{\sqcup_{x\in \{0,\cdots, m-1\}^{n}}[x,x+1)^{n}}h(z)~dz\\
&=\frac{1}{m^{n}}\sum_{x\in \{0,\cdots, m-1\}^{n}}\int_{[x,x+1)^{n}}h(z)~dz\\
&=\frac{1}{m^{n}}\sum_{x\in\{0,\ldots,m-1\}^n}\int_{[0,1)^n}h(x+z)\,dz.
\end{split}
\end{equation}

Let $x$ be uniform on $\{0,\ldots,m-1\}^n$, and let $U$, $W$ be independent
of $x$ and of each other such that $U$ is uniform on $[0,1)^n$ and $W$ is as in
Lemma \ref{lem:kernel}. Then by the definition of $F$, it is clear that
\[
 F(x+U)=f(x),\qquad
 F(x+U+W)=f(x+\lfloor U+W\rfloor).
\]
We claim here that the variable $\lfloor U+W\rfloor$ is a random variable which is independent of $x$ such that the distribution of $\lfloor U+W\rfloor$ is uniform on $\{-1,0,1\}^n$. Indeed, for each $j\in \{1,\cdots,n\}$, we have
\begin{align*}
\mathbb{P}\{\theta: \lfloor U_{j}+W_{j}\rfloor(\theta)=1\}&=\int_{[0,1)}\int_{\mathbb{R}}\psi(t)\chi_{[1,\infty)}(\omega+t)~dt~d\omega\\
&=\int_{\mathbb{R}}\psi(t)\left(\int_{[0,1)}\chi_{[1,\infty)}(\omega+t)~d\omega\right)dt\\
&=\int_{0}^{\infty}t\psi(t)dt=\frac{1}{3}.
\end{align*}
For other cases, it is also easy to compute that
\[
\mathbb{P}\{\theta:\lfloor U_{j}+W_{j} \rfloor(\theta)=0\}=\mathbb{P}\{\theta:\lfloor U_{j}+W_{j}\rfloor(\theta)=-1\}=\frac{1}{3}.
\]
This completes the proof of our claim. Hence, \eqref{eq:cellintegration} yields that
\begin{equation}\label{eq:shortidentity}
 \int_{\T_m^n}\E_W\norm{F(z+W)-F(z)}{1}^2\,d\nu(z)=\E_{(x,\varepsilon)\in\Z_m^n\times\{-1,0,1\}^{n}}\norm{f(x+\varepsilon)-f(x)}1^2.
\end{equation}

We now apply the smoothing approximation argument as follows. Let $g$ be defined by \eqref{eq:smoothing}. By Proposition \ref{prop:smoothing}, the function $g$ is $C^{1}$ and periodic. For every $z$, the Cauchy--Schwarz inequality gives
\[
 \bigl(\E_W\norm{F(z+W)-F(z)}{1}\bigr)^2\leq\E_W\left(\norm{F(z+W)-F(z)}{1}^2\right).
\]
Hence, applying Proposition \ref{prop:smoothing} and \eqref{eq:shortidentity} we obtain that
\begin{equation}\label{eq:globalapprox}
\begin{split}
\norm{F-g}{L_{2}(\mathbb{T}^{n}_{m};S_{1})}&\leq 3\left(\int_{\mathbb{T}^{n}_{m}}\left(\mathbb{E}_{W}\left\|F(z+W)-F(z)\right\|_{1}\right)^{2}~d\nu(z)\right)^{1/2}\\
&\leq 3\left(\int_{T^{n}_{m}}\E_{W}\left\|F(z+W)-F(z)\right\|^{2}_{1}~d\nu(z)\right)^{1/2}\\
&=3\left(\E_{(x,\varepsilon)\in\Z_m^n\times\{-1,0,1\}^{n}}\norm{f(x+\varepsilon)-f(x)}1^2\right)^{1/2},
\end{split}
\end{equation}
and
\begin{equation}\label{eq:globalgradient}
\left(\sum_{j=1}^n\norm{\partial_jg}{L_2(\mathbb{T}^{n}_{m};S_1)}^2\right)^{1/2}\leq6\sqrt{70}\left(\E_{(x,\varepsilon)\in\Z_m^n\times\{-1,0,1\}^{n}}\norm{f(x+\varepsilon)-f(x)}1^2\right)^{1/2}.
\end{equation}

For an integer vector $h\in \mathbb{N}^{n}$ and the fact $F(x+u+h)=f(x+h)$ for each $u\in [0,1)^{n}$, we have
\begin{equation}\label{representation 1}
\begin{split}
\norm{F(\cdot+h)-F}{L_2(\mathbb{T}^{n}_{m};S_1)}^2&=\int_{T^{n}_{m}}\|F(z+h)-F(z)\|^{2}_{1}~d\nu(z)\\
&=\E_{x\in \mathbb{Z}^{n}_{m}}\int_{[0,1)^{n}}\|F(x+u+h)-F(x+u)\|^{2}_{1}~d u\\
&=\E_{x\in\Z_m^n}\norm{f(x+h)-f(x)}{1}^{2}.
\end{split}
\end{equation}
Since $m$ is even, it follows from \eqref{representation 1} that
\begin{equation}\label{eq:longidentity}
 \left(\sum_{j=1}^{n}\E_{x\in \mathbb{Z}^{n}_{m}}\norm{f(x+\tfrac m2e_j)-f(x)}{1}^2\right)^{1/2}=\left(\sum_{j=1}^n
   \norm{F(\cdot+\tfrac m2e_j)-F}{L_2(\mathbb{T}^{n}_{m}; S_1)}^2\right)^{1/2}.
\end{equation}
For each $j$, by $g\in C^{1}(\mathbb{R}^{n};M^{\sa}_{d})$, the fundamental theorem of calculus gives
\[
g(z+\tfrac m2e_j)-g(z)=\int_0^{m/2}\partial_jg(z+te_j)\,dt.
\]
Hence, the Minkowski inequality yields
\begin{equation}\label{fundamental 1}
\begin{split}
\norm{g(\cdot+\tfrac m2e_j)-g}{L_2(\mathbb{T}^{n}_{m};S_1)}\leq\int_0^{m/2}\norm{\partial_jg(\cdot+te_j)}{L_2(\mathbb{T}^{n}_{m};S_1)}\,dt=\frac m2\norm{\partial_jg}{L_2(\mathbb{T}^{n}_{m}; S_1)}.
\end{split}
\end{equation}
where we used the fact that uniform measure on $\mathbb{T}^{n}_{m}$ is translation invariant. Applying \eqref{eq:longidentity}, we obtain
\begin{equation}\label{triangle 1}
\begin{split}
&\left(\sum_{j=1}^{n}\E_{x\in \mathbb{Z}^{n}_{m}}\norm{f(x+\frac{m}{2}e_{j})-f(x)}{1}^{2}\right)^{1/2}\\
\leq&\left(\sum_{j=1}^{n}\norm{F(\cdot+\frac{m}{2}e_{j})-g(\cdot+\frac{m}{2}e_{j})}{L_{2}(\mathbb{T}^{n}_{m};S_{1})}^{2}\right)^{1/2}\\
&+\left(\sum_{j=1}^{n}\norm{g(\cdot+\frac{m}{2}e_{j})-g}{L_{2}(\mathbb{T}^{n}_{m};S_{1})}^{2}\right)^{1/2}+\sqrt{n}\norm{F-g}{L_{2}(\mathbb{T}^{n}_{m};S_{1})}\\
=&2\sqrt{n}\norm{F-g}{L_{2}(\mathbb{T}^{n}_{m};S_{1})}+\left(\sum_{j=1}^{n}\norm{g(\cdot+\frac{m}{2}e_{j})-g}{L_{2}(\mathbb{T}^{n}_{m};S_{1})}^{2}\right)^{1/2}.
\end{split}
\end{equation}
Due to \eqref{eq:globalapprox}, we now apply Proposition \ref{prop:smoothing} and \eqref{eq:globalgradient} to \eqref{triangle 1} and obtain that
\begin{equation}\label{self-adjoint case}
\begin{split}
&\left(\sum_{j=1}^{n}\E_{x\in \mathbb{Z}^{n}_{m}}\norm{f(x+\frac{m}{2}e_{j})-f(x)}{1}^{2}\right)^{1/2}\\
\leq& 2\sqrt{n}\norm{F-g}{L_{2}(\mathbb{T}^{n}_{m};S_{1})}+3\sqrt{70}m\left(\E_{(x,\varepsilon)\in\Z_m^n\times\{-1,0,1\}^{n}}\norm{f(x+\varepsilon)-f(x)}1^2\right)^{1/2}\\
\leq & (6\sqrt{n}+3\sqrt{70}m)\left(\E_{(x,\varepsilon)\in \mathbb{Z}^{n}_{m}\times \{-1,0,1\}^{n}}\norm{f(x+\varepsilon)-f(x)}{1}^{2}\right)^{1/2},
\end{split}
\end{equation}
for every $f:\mathbb{Z}^{n}_{m}\to M^{\sa}_{d}$. This completes the proof for the self-adjoint matrices setting.

{\bf Step II. Finite dimensional case.} For a matrix $A\in M_{d}$, we define the Hermitian dilation of $A$ by
\[
\widetilde{A}\coloneqq\frac{1}{2}
\begin{pmatrix}
0&A\\ A^*&0
\end{pmatrix}\in M_{2d}^{\sa}
\]
It is easy to see that
\[
\norm{\widetilde{A}}{1}=\frac12\Tr(AA^*)^{1/2}+\frac12\Tr(A^*A)^{1/2}=\norm A1.
\]
By the Hermitian dilation of matrices, we now define the Hermitian dilation of a function $f:\mathbb{Z}^{n}_{m}\to M_{d}$ as follows
\[
 \widetilde f(x)=\frac{1}{2}
\begin{pmatrix}
0&f(x)\\ f(x)^*&0
\end{pmatrix}\in M_{2d}^{\sa}, \qquad x\in \mathbb{Z}^{n}_{m}.
\]

Hence, for a given $f:\mathbb{Z}^{n}_{m}\to M_{d}$, we apply \eqref{self-adjoint case} to $\widetilde{f}$ to obtain
\begin{align*}
&\left(\sum_{j=1}^{n}\E_{x\in \mathbb{Z}^{n}_{m}}\norm{f(x+\frac{m}{2}e_{j})-f(x)}{1}^{2}\right)^{1/2}\\
=&\left(\sum_{j=1}^{n}\E_{x\in \mathbb{Z}^{n}_{m}}\norm{\widetilde{f}(x+\frac{m}{2}e_{j})-\widetilde{f}(x)}{1}^{2}\right)^{1/2}\\
\leq & (6\sqrt{n}+3\sqrt{70}m)\left(\E_{(x,\varepsilon)\in \mathbb{Z}^{n}_{m}\times \{-1,0,1\}^{n}}\norm{\widetilde{f}(x+\varepsilon)-\widetilde{f}(x)}{1}^{2}\right)^{1/2}\\
=&(6\sqrt{n}+3\sqrt{70}m)\left(\E_{(x,\varepsilon)\in \mathbb{Z}^{n}_{m}\times \{-1,0,1\}^{n}}\norm{f(x+\varepsilon)-f(x)}{1}^{2}\right)^{1/2}.
\end{align*}
This proves the assertion for mappings taking values in arbitrary finite dimensional matrices.

{\bf Step III. Infinite dimensional case via approximation.}
For each $x\in\Z_m^n$, by the singular value decomposition of $f(x)$, we have
\[
 f(x)=\sum_{k=1}^\infty s_k(x)u_k(x)\otimes v_k(x)
\]
with
\[
\|f(x)\|_{1}=\sum_{k=1}^\infty s_k(x)<\infty,
\]
and $u\otimes v$ denotes the operator $h\mapsto\langle h,v\rangle u$. For each $N\in \mathbb{N}$, let
\[
 f_N(x)=\sum_{k=1}^N s_k(x)u_k(x)\otimes v_k(x),\qquad
 \eta_N=\max_{x\in\Z_m^n}\sum_{k>N}s_k(x).
\]
By the definition of the Schatten trace class norm, it follows that $\norm{f_N(x)-f(x)}1=\sum_{k>N}s_k(x)\to 0$ as $N\to \infty$. Since $\mathbb{Z}^{n}_{m}$ is a finite set, we have
$\eta_N\to0$. By {\bf Step II}, it follows that for each $N\in \mathbb{N}$
\begin{equation}\label{finite version}
\begin{split}
&\left(\sum_{j=1}^{n}\E_{x\in \mathbb{Z}^{n}_{m}}\norm{f_{N}(x+\frac{m}{2}e_{j})-f_{N}(x)}{1}^{2}\right)^{1/2}\\
\leq&(6\sqrt{n}+3\sqrt{70}m)\left(\E_{(x,\varepsilon)\in \mathbb{Z}^{n}_{m}\times \{-1,0,1\}^{n}}\norm{f_{N}(x+\varepsilon)-f_{N}(x)}{1}^{2}\right)^{1/2}
\end{split}
\end{equation}
Taking $N\to \infty$ on the both sides of \eqref{finite version} yields the desired result.
\end{proof}

We now conclude this paper with the following corollary which settles the metric cotype problem on the Schatten trace class posed by Mendel and Naor in \cite[Section 8]{MN}.

\begin{corollary}\label{cor:scale}
If $m\geq\sqrt n$ is even and $m\geq4$, then for every $f:\Z_m^n\to S_1$ we have
\begin{align*}
&\left(\sum_{j=1}^{n}\E_{x\in \mathbb{Z}^{n}_{m}}\norm{f(x+\frac{m}{2}e_{j})-f(x)}{1}^{2}\right)^{1/2}\\
\leq&32m\left(\E_{(x,\varepsilon)\in \mathbb{Z}^{n}_{m}\times \{-1,0,1\}^{n}}\norm{f(x+\varepsilon)-f(x)}{1}^{2}\right)^{1/2}.
\end{align*}
Conversely, if $C>0$ and even integers $\ell_{n} \geq4$ satisfy
\begin{align*}
&\left(\sum_{j=1}^{n}\E_{x\in \mathbb{Z}^{n}_{\ell_{n}}}\norm{f(x+\frac{\ell_{n}}{2}e_{j})-f(x)}{1}^{2}\right)^{1/2}\\
\leq&C\ell_{n}\left(\E_{(x,\varepsilon)\in \mathbb{Z}^{n}_{\ell_{n}}\times \{-1,0,1\}^{n}}\norm{f(x+\varepsilon)-f(x)}{1}^{2}\right)^{1/2}
\end{align*}
for every $n\geq 1$ and every $f:\mathbb{Z}^{n}_{\ell_{n}}\to S_{1}$, then
\begin{equation}
 \ell_n\geq\frac{\sqrt n}{C\sqrt{1-3^{-n}}}\geq\frac{\sqrt n}{C}.
 \label{eq:necessity}
\end{equation}
\end{corollary}

\section*{Acknowledgment}

This work was supported by the National Natural Science Foundation of China (Grant Nos. 12125109, W2411005 \& 12671165); the Natural Science Foundation of Hunan Province (Grant Nos: 2025ZYJ002, 2024JJ1010 \& 2024RC3040); the Scientific Research Fund of Hunan Provincial Education Department (Grant No. 25B0008).

During the preparation of this work, the authors used GPT-6 Astra to assist with language editing and refinement of mathematical arguments. The author takes full responsibility for the content of this paper.

\end{document}